\documentclass[12pt]{article}
\usepackage[a4paper,width=6.8in,height=8.8in]{geometry}
\usepackage{amsmath,amssymb}

\usepackage{amsmath,amssymb,mathtools,amsthm,
	textcomp,mathscinet}
\usepackage{amsfonts,graphicx,enumerate,subcaption,placeins}
\usepackage[hidelinks]{hyperref}
\usepackage{setspace}
\usepackage[utf8]{inputenc}
\usepackage[english]{babel}
\usepackage{booktabs,multirow}
\usepackage{blkarray}
\usepackage{lscape}
\usepackage[T1]{fontenc}
\usepackage{authblk,color}
\usepackage[english]{babel}
\usepackage{amsfonts, amstext, amsthm, booktabs, dcolumn}
\usepackage{graphicx}
\usepackage{float}
\usepackage{caption}
\usepackage{centernot}
\usepackage{times}
\usepackage[nottoc]{tocbibind}
\usepackage[mathscr]{eucal}

\definecolor{gr}{rgb}{0.7, 0.0, 0.15}

\newtheorem{theorem}{\bf Theorem}[section]

\newtheorem{corollary}{\bf Corollary}[section]
\newtheorem{remark}{\bf Remark}[section]
\newtheorem{remarks}{\bf Remarks}[section]
\newtheorem{lemma}{\bf Lemma}[section]

\newcommand\ceil[1]{\left\lceil#1\right\rceil}

\numberwithin{equation}{section}
\begin{document}
\title{Bounds on Negative Binomial Approximation to Call Function}
\author[ ]{{\Large Amit N. Kumar}}
\affil[ ]{Department of Mathematical Sciences}
\affil[ ]{Indian Institute of Technology (BHU)}
\affil[ ]{Varanasi-221005, India.}
\affil[ ]{Email: amit.kumar2703@gmail.com}
\date{}
\maketitle

\begin{abstract}
\noindent
In this paper, we develop Stein's method for negative binomial distribution using call function defined by $f_z(k)=(k-z)^+=\max\{k-z,0\}$, for $k\ge 0$ and $z \ge 0$. We obtain error bounds between $\mathbb{E}[f_z(\text{N}_{r,p})]$ and $\mathbb{E}[f_z(V)]$, where $\text{N}_{r,p}$ follows negative binomial distribution and $V$ is the sum of locally dependent random variables, using certain conditions on moments. We demonstrate our results through an interesting application, namely, collateralized debt obligation (CDO), and compare the bounds with the existing bounds. 
\end{abstract}

\noindent

\section{Introduction}\label{10:sec1}
The call function is a non-negative real-valued function of the form
\begin{align}
f_z(k)=(k-z)^+=\max\{k-z,0\}, \quad \text{for }k\ge 0 \text{ and } z\ge 0.\label{10:cf}
\end{align}
It has been used in several areas of probability and statistics, for example, finance, risk theory, and derivative pricing, among many others. In particular, it has been successfully applied to the collateralized debt obligation (CDO). For more details, see Karoui and Jiao \cite{JK}, Karoui {\em et al.} \cite{JKK},  Hull and White \cite{HW}, Neammanee and Yonghint \cite{NY}, Yonghint {\em et al.} \cite{YNC2}, and references therein.\\
For a random variable (rv) $W$, the study of $\mathbb{E}[f_z(W)]$ plays an important role in many real-life applications. For example, if $W$ is the sum of Bernoulli random variables (rvs) then $\mathbb{E}[f_z(W)]$ is used to compute the mean value of total percentage loss for each tranche in CDO (see Neammanee and Yonghint \cite{NY}, and Yonghint {\em et al.} \cite{YNC2} for details). Also, if $W$ has a complicated structure, for example, $W$ is the sum of locally dependent or independent (but non-identical) rvs, then $\mathbb{E}[f_z(W)]$ becomes difficult to compute in practice. In such cases, an approximation to standard and easy-to-use distribution is of interest. Approximation to call function has been studied by several authors in the literature, for example, Poisson approximation has been studied by Neammanee and Yonghint \cite{NY}, and Yonghint {\em et al.} \cite{YNC2}, and Normal approximation has been studied by Karoui and Jiao \cite{JK}, and Karoui {\em et al.} \cite{JKK}. \\
In this paper, we study negative binomial (NB) approximation to call function using certain conditions on moments. The main advantage of NB distribution over Poisson distribution is the extra flexibility parameter that builds our bounds more shaper compare to the existing bounds for Poisson approximation. Throughout this paper, let $\text{N}_{r,p}$ follow NB distribution with probability mass function
\begin{align}
\mathbb{P}(\text{N}_{r,p}=k)=\binom{r+k-1}{k}p^{r} q^{k}, \quad k\in \mathbb{Z}_+,\label{10:nb}
\end{align}
where $r>1$, $q=1-p \in (0,1)$ and $\mathbb{Z}_+=\{0,1,2,\ldots\}$, the set of non-negative integers. From Neammanee and Yonghint \cite{NY}, and Yonghint {\em et al.} \cite{YNC2}, We observe that the call function can be studied under a locally dependent or independent setup. Therefore, we consider the following locally dependent structure that can be used for both cases. \\
Let $J$ be a finite subset of $\mathbb{N}=\{1,2,\ldots\}$, the set of all positive integers, and $\{\zeta_i\}_{i\in J}$ be a collection of non-negative rvs. For each $i$, let $i\in A_i\subseteq B_i\subset J$ be such that $\zeta_i$ is independent of $\zeta_{A_i^c}$ and $\zeta_{A_i}$ is independent of $\zeta_{B_i^c}$, where $\zeta_A$ is the collection of rvs $\{\zeta_i\}_{i\in A}$ and $A^c$ denotes the complement of the set $A$. See Section 3 of R\"{o}llin \cite{RO2008} and Section 2 of Kumar \cite{k2021} for a similar type of locally dependent structure.  Define
\begin{align}
V=\sum_{i\in J}\zeta_i,\label{10:v}
\end{align}
the sum of locally dependent rvs. Note that if $A_i=B_i=\{i\}$ then $V$ is the sum of independent rvs. Throughout this paper, we let $\zeta_{A}=\sum_{i\in A} \zeta_i$, for a set $A\subset J$, and $\mathscr{D}(W):=2d_{TV}(W,W+1)$, for a rv $W$, where $d_{TV}(X,Y)$ denotes the total variation distance between $X$ and $Y$. In this paper, our aim is to study the proximity between $\mathbb{E}[(V-z)^+]$ and $\mathbb{E} [(\text{N}_{r,p}-z)^+]$. That is, our interest is to obtain the upper bound for 
\begin{align}
\left|\mathbb{E}[(V-z)^+]-\mathbb{E} [(\text{N}_{r,p}-z)^+]\right|.\label{10:bd}
\end{align}
We use Stein's method to obtain the bound for the above expression discussed in Section \ref{10:sec2}. 

\noindent
This paper is organized as follows. In Section \ref{10:sec2}, we develop Stein's method for NB distribution using the call function. In Section \ref{10:sec3}, we obtain uniform and non-uniform bounds for the expression given in \eqref{10:bd} and compare our results with the existing results. In Section \ref{10:sec3}, we give an application of our results to CDO and give some numerical comparisons. Finally, in Appendix A, we give some inequalities and their proofs that are useful to develop Stein's method for NB distribution.

\section{Stein's Method}\label{10:sec2}
Stein's method (Stein \cite{stein1972}) is a tool for obtaining error bounds between two probability distributions. This method is mainly based on obtaining the solution of the Stein equation given by
\begin{align}
\mathscr{A} g(k)=f(k)-\mathbb{E} f(X),\quad \text{for }k\in \mathbb{Z}_+,\label{10:se}
\end{align}
where $\mathscr{A}$ is a Stein operator for a rv $X$ such that $\mathbb{E}[\mathscr{A} g(X)]=0$, $f$ and $g$ are real-valued bounded functions on $\mathbb{Z}_+$. Stein's method has been developed for NB distribution by Brown and Phillips \cite{BP} and Barbour {\em et al.} \cite{BGX} for total variation distance and Wasserstein distance, respectively. In this section, we develop Stein's method for NB distribution when $f$ is a call function, defined in \eqref{10:cf}, which is used to obtain upper bounds for the expression given in \eqref{10:bd}. The NB approximation via Stein's method has been studied by several authors such as Barbour {\em et al.} \cite{BGX}, Brown and Phillips \cite{BP}, Vellaisamy {\em et al.} \cite{VUC}, Wang and Xia \cite{WX}, Kumar and Upadhye \cite{KU}, among many others.\\
Next, let $X= \text{N}_{r,p}$ and $f=f_z$, defined in \eqref{10:cf}, then the Stein equation \eqref{10:se} leads to
\begin{align}
\mathscr{A} g(k)=(k-z)^+-\mathbb{E} [(\text{N}_{r,p}-z)^+],\quad\text{for } k\in \mathbb{Z}_+\text{ and }z\ge 0.\label{10:ses}
\end{align}
Also, let $g=g_z$ be the solution of the above equation. Now, replacing $k$ by $V$ in \eqref{10:ses} and taking expectation, we get
\begin{align}
\left|\mathbb{E}[\mathscr{A}g_z(V)]\right|=\left|\mathbb{E}[(V-z)^+]-\mathbb{E} [(\text{N}_{r,p}-z)^+]\right|. \label{10:mb}
\end{align}
Therefore, to obtain the upper bound for the expression given in \eqref{10:bd}, it is enough to obtain the upper bound for $\left|\mathbb{E}[\mathscr{A} g_z(V)]\right|$. \\
Next, the Stein operator of $\text{N}_{r,p}$ is given by
\begin{align}
\mathscr{A} g(k)=q(r+k)g(k+1)-kg(k),\quad \text{for }k\in \mathbb{Z}_+.\label{10:so}
\end{align}
See Lemma 1 of Brown and Phillips \cite{BP} for details. Substituting \eqref{10:so} in \eqref{10:ses}, we get
\begin{align}
q(r+k)g(k+1)-kg(k)=(k-z)^+-\mathbb{E} [(\text{N}_{r,p}-z)^+].\label{10:sess}
\end{align} 
It can be easily verified that the solution of \eqref{10:sess} is
\begin{align}
g_z(k)&=\left\{\begin{array}{ll}
0 & \text{if } k=0;\\
-\displaystyle{\sum_{j=k}^{\infty}\frac{r(r+1)\ldots (r+j-1)}{r(r+1)\ldots (r+k-1)}\frac{(k-1)!}{j!}q^{j-k}[(j-z)^+-\mathbb{E} [(\text{N}_{r,p}-z)^+]]} & \text{if }k\ge1.
\end{array}\right.\label{10:sol}
\end{align}
For more details, see Section 2 of Kumar {\em et al.} \cite[p. 4]{KUV} with appropriate changes. Now, we move to obtain uniform and non-uniform upper bound for $|g_z(\cdot)|$ and $|\Delta g_z(\cdot)|$, where $\Delta g(k)=g(k+1)-g(k)$ denotes the first forward difference operator. Some of the proofs of the following results are similar to the proofs given by Neammanee and Yonghint \cite{NY}.

\begin{lemma}\label{10:le2}
For $k\ge 0$ and $z\ge 0$, $g_z$ defined in \eqref{10:sol} satisfies the following:
\begin{enumerate}
\item[(i)] $|g_z(k)|\le p^{-(r+1)}$.
\item[(ii)] $|\Delta g_z(k)|\le 2p^{-(r+1)}-p^{-1}$.
\end{enumerate}
\end{lemma}
\begin{proof}
\begin{enumerate}
\item[(i)] As $g_z(0)=0$, it is enough to prove the result for $k\ge 1$. Consider
\begin{align}
0&<\sum_{j=k}^{\infty}\frac{r(r+1)\ldots (r+j-1)}{r(r+1)\ldots (r+k-1)}\frac{(k-1)!}{j!}q^{j-k}(j-z)^+\nonumber\\
&\le 1+\sum_{j=k+1}^{\infty}\frac{r(r+1)\ldots (r+j-1)}{r(r+1)\ldots (r+k-1)}\frac{(k-1)!}{(j-1)!}q^{j-k}\nonumber\\
&=1+\sum_{j=k+1}^{\infty} \frac{(r+k) \ldots (r+j-1)}{k(k+1)\ldots (j-1)}q^{j-k}\nonumber\\
&=1+\sum_{j=1}^{\infty}\frac{(r+k)\ldots (r+j+k-1)}{k(k+1)\ldots (j+k-1)}q^j\nonumber\\
&\le p^{-(r+1)},\quad \text{(using Lemma \ref{10:po2}(i)).}\label{10:in1}
\end{align}
Next, consider
\begin{align}
0&<\sum_{j=k}^{\infty}\frac{r(r\hspace{-0.05cm}+\hspace{-0.05cm}1)\ldots (r\hspace{-0.05cm}+\hspace{-0.05cm}j\hspace{-0.05cm}-\hspace{-0.05cm}1)}{r(r\hspace{-0.05cm}+\hspace{-0.05cm}1)\ldots (r\hspace{-0.05cm}+\hspace{-0.05cm}k\hspace{-0.05cm}-\hspace{-0.05cm}1)}\frac{(k\hspace{-0.05cm}-\hspace{-0.05cm}1)!}{j!}q^{j-k}\hspace{-0.07cm}\le 1+\sum_{j=k+1}^{\infty}\frac{r(r\hspace{-0.05cm}+\hspace{-0.05cm}1)\ldots (r\hspace{-0.05cm}+\hspace{-0.05cm}j\hspace{-0.05cm}-\hspace{-0.05cm}1)}{r(r\hspace{-0.05cm}+\hspace{-0.05cm}1)\ldots (r\hspace{-0.05cm}+\hspace{-0.05cm}k\hspace{-0.05cm}-\hspace{-0.05cm}1)}\frac{(k\hspace{-0.05cm}-\hspace{-0.05cm}1)!}{j!}q^{j-k}\nonumber\\
&= 1+\sum_{j=k+1}^{\infty}\frac{(r+k)\ldots (r+j-1)}{k(k+1)\ldots j}q^{j-k}=1+\sum_{j=1}^{\infty}\frac{(r+k)\ldots (r+j+k-1)}{k(k+1)\ldots (j+k)}q^{j}\nonumber\\
&\le 1+\sum_{j=1}^{\infty}\frac{(r+k)\ldots (r+j+k-1)}{(k+1)\ldots (j+k)}q^{j}\le  \frac{p^{-r}-1}{rq}, \quad \text{(using Lemma \ref{10:po2}(ii))}.\label{10:rs4}
\end{align}
Therefore, from Lemma \ref{10:po1}(i), we have
\begin{align}
0<\sum_{j=k}^{\infty}\frac{r(r+1)\ldots (r+j-1)}{r(r+1)\ldots (r+k-1)}\frac{(k-1)!}{j!}q^{j-k}\mathbb{E} [(\text{N}_{r,p}-z)^+]\le p^{-(r+1)}-p^{-1}.\label{10:in2}
\end{align}
Hence, from \eqref{10:sol}, \eqref{10:in1}, and \eqref{10:in2}, we get
\begin{align*}
|g_z(k)|=\left|\sum_{j=k}^{\infty}\frac{r(r+1)\ldots (r+j-1)}{r(r+1)\ldots (r+k-1)}\frac{(k-1)!}{j!}q^{j-k}[(j-z)^+-\mathbb{E} [(\text{N}_{r,p}-z)^+]]\right|\le p^{-(r+1)}.
\end{align*}
This proves (i).
\item[(ii)] Note that, for $k=0$,
\begin{align*}
|\Delta g_z(0)|=|g_z(1)|\le p^{-(r+1)}\le 2p^{-(r+1)}-p^{-1}.
\end{align*}
Now, we prove the result for $k\ge 1$. Let
\begin{align*}
A_1(k)&=\sum_{j=k}^{\infty}\frac{r(r+1)\ldots (r+j-1)}{r(r+1)\ldots (r+k-1)}\frac{(k-1)!}{j!}q^{j-k}(j-z)^+\\
&~~~-\sum_{j=k+1}^{\infty}\frac{r(r+1)\ldots (r+j-1)}{r(r+1)\ldots (r+k)}\frac{k!}{j!}q^{j-k-1}(j-z)^+
\end{align*}
and
\begin{align*}
A_2(k)&=\sum_{j=k+1}^{\infty}\frac{r(r+1)\ldots (r+j-1)}{r(r+1)\ldots (r+k)}\frac{k!}{j!}q^{j-k-1}\mathbb{E} [(\text{N}_{r,p}-z)^+]\\
&~~~-\sum_{j=k}^{\infty}\frac{r(r+1)\ldots (r+j-1)}{r(r+1)\ldots (r+k-1)}\frac{(k-1)!}{j!}q^{j-k}\mathbb{E} [(\text{N}_{r,p}-z)^+].
\end{align*}
Then
\begin{align*}
\Delta g_z(k)&=g_z(k+1)-g_z(k)=A_1(k)+A_2(k),
\end{align*}
Hence, using \eqref{10:in1} and \eqref{10:in2}, we have
\begin{align*}
|\Delta g_z(k)|\le |A_1(k)|+|A_2(k)|\le 2p^{-(r+1)}-p^{-1}.
\end{align*}
This proves (ii).
\end{enumerate}
\end{proof}

\begin{lemma}\label{10:le3}
For $k \ge 1$ and $z> 1$, $g_z$ defined in \eqref{10:sol} satisfies the following:
\begin{align*}
|\Delta g_z(k)|&\le \left\{
\begin{array}{ll}
\vspace{0.2cm}
\displaystyle{\frac{1}{z}\left(2p^{-(r+1)}-p^{-1}\right) }& \text{if } k\ge z;\\
\vspace{0.2cm}
\displaystyle{\frac{1}{z}\left((1+q^{-1})p^{-(r+2)}-p^{-2}\right)} & \text{if } 2\le k<z;\\
\displaystyle{\frac{(r+1)}{z}\left(2p^{-(r+2)}-p^{-2}\right)} & \text{if }k=1.
\end{array}\right.
\end{align*}
\end{lemma}
\begin{proof}
Let $k\ge z$. First, consider
\begin{align*}
A_1(k)&=\sum_{j=k}^{\infty}\frac{r(r+1)\ldots (r+j-1)}{r(r+1)\ldots (r+k)}\frac{(k-1)!}{(j+1)!}q^{j-k}[(r+k)(j+1)(j-z)^+-k(r+j)(j+1-z)^+]\\
&=\sum_{j=k}^{\infty}\frac{r(r+1)\ldots (r+j-1)}{r(r+1)\ldots (r+k)}\frac{(k-1)!}{(j+1)!}q^{j-k}[(r+k)(j+1)(j-z)-k(r+j)(j+1-z)].
\end{align*}
Observe that
\vspace{-0.41cm}
\begin{align}
|(r+k)(j+1)(j-z)&-k(r+j)(j+1-z)|\nonumber\\
&=|r(j+1)(j-k)-r(j-k)z-(k+r)z|\nonumber\\
&\le |r(j+1)(j-k)-r(j-k)z|+(k+r)z\nonumber\\
&=r(j+1)(j-k)-(r(j-k)-k-r)z\nonumber\\
&\le \left\{
\begin{array}{ll}
(k+r)z & \text{if } j=k;\\
r(j+1)(j-k) & \text{if }j>k.
\end{array}\right.\label{10:ppp}
\end{align}
Therefore,
\vspace{-0.41cm}
\begin{align}
|A_1(k)|&\le \frac{z}{k(k+1)}+r\sum_{j=k+1}^{\infty}\frac{r(r+1)\ldots (r+j-1)}{r(r+1)\ldots (r+k)}\frac{(k-1)! (j-k)}{j!}q^{j-k}\nonumber\\
&\le \frac{z}{k(k+1)}+\frac{r}{k}\sum_{j=k+1}^{\infty}\frac{r(r+1)\ldots (r+j-1)}{r(r+1)\ldots (r+k)}\frac{k!}{(j-1)!}q^{j-k}\nonumber\\
&= \frac{z}{k(k+1)}+\frac{rq}{k}+\frac{r}{k}\sum_{j=k+2}^{\infty}\frac{r(r+1)\ldots (r+j-1)}{r(r+1)\ldots (r+k)}\frac{k!}{(j-1)!}q^{j-k}\nonumber\\
&= \frac{z}{k(k+1)}+\frac{rq}{k}+\frac{r}{k}\sum_{j=k+2}^{\infty}\frac{(r+k+1)\ldots (r+j-1)}{(k+1)\ldots (j-1)}q^{j-k}\nonumber\\
&= \frac{1}{z}\left(1+rq+r\sum_{j=2}^{\infty}\frac{(r+k+1)\ldots (r+j+k-1)}{(k+1)\ldots (j+k-1)}q^{j}\right)\nonumber\\
&\le  \frac{p^{-(r+1)}}{z}, \quad \text{(using Lemma \ref{10:po2}(iii))}.\label{10:ak1}
\end{align}
Now, consider
\begin{align*}
\sum_{j=k}^{\infty}\frac{r(r+1)\ldots (r+j-1)}{r(r+1)\ldots (r+k-1)}\frac{(k-1)!}{j!}q^{j-k}&=\frac{1}{k}\sum_{j=k}^{\infty}\frac{r(r+1)\ldots (r+j-1)}{r(r+1)\ldots (r+k-1)}\frac{k!}{j!}q^{j-k}\\
&= \frac{1}{z}\left(1+\sum_{j=k+1}^{\infty}\frac{(r+k)\ldots (r+j-1)}{(k+1)\ldots j}q^{j-k}\right)\\
&= \frac{1}{z}\left(1+\sum_{j=1}^{\infty}\frac{(r+k)\ldots (r+j+k-1)}{(k+1)\ldots (j+k)}q^{j}\right)\\
& \le \frac{p^{-r}-1}{rqz},\quad \text{(using Lemma \ref{10:po2}(ii)).}
\end{align*}
Therefore, from Lemma \ref{10:po1}(i), we have
\begin{align}
\sum_{j=k}^{\infty}\frac{r(r+1)\ldots (r+j-1)}{r(r+1)\ldots (r+k-1)}\frac{(k-1)!}{j!}q^{j-k}\mathbb{E}[(\text{N}_{r,p}-z)^+]\le \frac{p^{-(r+1)}-p^{-1}}{z}.\label{10:l1}
\end{align} 
Hence, for $k\ge z$, from \eqref{10:ak1} and \eqref{10:l1}, we have
\begin{align*}
|\Delta g(k)|\le |A_1(k)|+|A_2(k)|\le \frac{1}{z}\left(2p^{-(r+1)}-p^{-1}\right).
\end{align*}
Next, let $k<z$ and consider
\begin{align}
|A_1(k)|\hspace{-0.1cm}&\le\hspace{-0.2cm}\sum_{j=\ceil{z}-1}^{\infty}\hspace{-0.2cm}\frac{r(r+1)\ldots (r+j-1)}{r(r+1)\ldots (r+k)}\frac{(k-1)!}{(j+1)!}q^{j-k}|(r+k)(j+1)(j-z)^+-k(r+j)(j+1-z)^+|\nonumber\\
&\le\frac{r(r+1)\ldots (r+\ceil{z}-1)}{r(r+1)\ldots (r+k)}\frac{k!}{(\ceil{z})!}(\ceil{z}-z)q^{\ceil{z}-1-k}\nonumber\\
&~~~+\frac{r}{\ceil{z}}\sum_{j=\ceil{z}}^{\infty}\frac{r(r+1)\ldots (r+j-1)}{r(r+1)\ldots (r+k)}\frac{(k-1)!}{(j-2)!}q^{j-k} \quad \text{(using \eqref{10:ppp})}\label{10:p2}\\
&\le\frac{r(r+1)\ldots (r+\ceil{z}-1)}{r(r+1)\ldots (r+k)}\frac{k!}{(\ceil{z})!}(\ceil{z}-z)q^{\ceil{z}-1-k}\nonumber\\
&~~~+\frac{r}{z}\frac{r(r+1)\ldots (r+\ceil{z}-1)}{r(r+1)\ldots (r+k)}\frac{(k-1)!}{(\ceil{z}-2)!}q^{\ceil{z}-k}\nonumber\\
&~~~+\frac{r}{z}\sum_{j=\ceil{z}+1}^{\infty}\frac{(r+k+1)\ldots (r+j-1)}{k(k+1)\ldots (j-2)}q^{j-k},\nonumber
\end{align}
\vspace{-0.001cm}
where $\ceil{z}$ is the smallest integer greater than or equal to $z$. At $k=\ceil{z}-1$, we have
\begin{align}
|A_1(k)|&\le \frac{1}{z}\left(1+rq+r\sum_{j=\ceil{z}+1}^{\infty}\frac{(r+\ceil{z})\ldots (r+j-1)}{(\ceil{z}-1)(\ceil{z})\ldots (j-2)}q^{j-\ceil{z}+1}\right)\nonumber\\
&=\frac{1}{z}\left(1+rq+r\sum_{j=\ceil{z}+1-k}^{\infty}\frac{(r+\ceil{z})\ldots (r+j+k-1)}{(\ceil{z}-1)(\ceil{z})\ldots (j+k-2)}q^{j+k-\ceil{z}+1}\right)\nonumber\\
&=\frac{1}{z}\left(1+rq+r\sum_{j=2}^{\infty}\frac{(r+k+1)\ldots (r+j+k-1)}{k(k+1)\ldots (j+k-2)}q^{j}\right)\nonumber\\
&=\frac{1}{z}\left(1+rq+rq\sum_{j=1}^{\infty}\frac{(r+k+1)\ldots (r+j+k)}{k(k+1)\ldots (j+k-1)}q^{j}\right)\label{10:p1}\\
&\le \frac{p^{-(r+2)}}{z}\quad \text{(using Lemma \ref{10:po2}(iv)).}\label{10:p8}
\end{align}
Now, let $k <\ceil{z}-1$. From \eqref{10:p2}, we have
\begin{align}
|A_1(k)|&\le \frac{r(r+1)\ldots (r+\ceil{z}-1)}{r(r+1)\ldots (r+k)}\frac{k!}{(\ceil{z})!}(\ceil{z}-z)q^{\ceil{z}-1-k}\nonumber\\
&~~~+\frac{r}{\ceil{z}}\sum_{j=\ceil{z}}^{\infty}\frac{(r+k+1)\ldots (r+j-1)}{k(k+1)\ldots (j-2)}q^{j-k}\nonumber \\
&\le \frac{1}{z}\frac{(r+k+1)\ldots (r+\ceil{z}-1)}{k(k+1)\ldots (\ceil{z}-2)}(\ceil{z}-z)q^{\ceil{z}-1-k}\nonumber\\
&~~~+\frac{r+1}{z}\sum_{j=\ceil{z}}^{\infty}\frac{(r+k+1)\ldots (r+j)}{k(k+1)\ldots (j-1)}q^{j-k} \nonumber\\
&\le \frac{r+1}{z}\sum_{j=\ceil{z}-1}^{\infty}\frac{(r+k+1)\ldots (r+j)}{k(k+1)\ldots (j-1)}q^{j-k}\nonumber\\
&\le \frac{r+1}{z}\sum_{j=k+1}^{\infty}\frac{(r+k+1)\ldots (r+j)}{k(k+1)\ldots (j-1)}q^{j-k}\nonumber\\
&\le \frac{r+1}{z}\sum_{j=1}^{\infty}\frac{(r+k+1)\ldots (r+j+k)}{k(k+1)\ldots (j+k-1)}q^{j}\label{10:p3}\\
&\le \frac{p^{-(r+2)}}{qz},\quad \text{(using Lemma \ref{10:po2}(iv)).}\label{10:p14}
\end{align}
Next, for $k\ge 2$, consider
\begin{align}
\sum_{j=k}^{\infty}\frac{r(r+1)\ldots (r+j-1)}{r(r+1)\ldots (r+k-1)}\frac{(k-1)!}{j!}q^{j-k}&=\frac{1}{k}+\sum_{j=k+1}^{\infty}\frac{(r+k)\ldots (r+j-1)}{k(k+1)\ldots j}q^{j-k}\nonumber\\
&\le \frac{1}{2}+\sum_{j=1}^{\infty}\frac{(r+k)\ldots (r+j+k-1)}{k(k+1)\ldots (j+k)}q^{j}\nonumber\\
&\le \frac{p^{-r}-1}{r(r+1)q^2},\quad \text{(using Lemma \ref{10:po2}(v))}.\label{10:b11}
\end{align}
Therefore, from Lemma \ref{10:po1}(ii) and \eqref{10:b11}, we get
\begin{align}
\sum_{j=k}^{\infty}\frac{r(r+1)\ldots (r+j-1)}{r(r+1)\ldots (r+k-1)}\frac{(k-1)!}{j!}q^{j-k}\mathbb{E}[(\text{N}_{r,p}-z)^+]\le \frac{p^{-(r+2)}-p^{-2}}{z}.\label{10:b12}
\end{align}
Hence, for $k<z$, from \eqref{10:p8}, \eqref{10:p14}, and \eqref{10:b12}, we have
\begin{align*}
|\Delta g(k)|\le |A_1(k)|+|A_2(k)|\le \frac{1}{z}\left((1+q^{-1})p^{-(r+2)}-p^{-2}\right).
\end{align*}
Next, at $k=1$, from \eqref{10:p1}, we have
\begin{align}
|A_1(1)|\le \frac{1}{z}\left(1+rq+rq\sum_{j=1}^{\infty}\binom{r+j+1}{j}q^j\right)\le \frac{1}{z}\left(1+rqp^{-(r+2)}\right)\le \frac{(r+1)p^{-(r+2)}}{z}\label{10:p21}
\end{align}
and, at $k=1$, from \eqref{10:p3}, we have
\begin{align}
|A_1(1)|\le \frac{r+1}{z}\sum_{j=1}^{\infty}\binom{r+j+1}{j}q^j=\frac{(r+1)\left(p^{-(r+2)}-1\right)}{z}\le \frac{(r+1)p^{-(r+2)}}{z}.\label{10:p22}
\end{align}
Also, using Lemma \ref{10:po1}(ii), it can be easily verified that
\begin{align}
|A_2(1)|\le \frac{1}{r}\sum_{j=1}^{\infty}\frac{r(r+1)\ldots (r+j-1)}{j!}q^{j-1}\mathbb{E}[(\text{N}_{r,p}-z)^+]\le \frac{(r+1)\left(p^{-(r+2)}-p^{-2}\right)}{z}.\label{10:p23}
\end{align}
Hence, at $k=1$, from \eqref{10:p21}, \eqref{10:p22}, and \eqref{10:p23}, we have
\begin{align*}
|\Delta g(1)|\le |A_1(1)|+|A_2(1)|\le \frac{(r+1)}{z}\left(2p^{-(r+2)}-p^{-2}\right).
\end{align*}
This proves the result.
\end{proof}

\begin{remark} \label{10:re1}
From Lemma \ref{10:le3}, a rather crude uniform bound is given by
\begin{align}
\|\Delta g_z\|\le \vartheta_{r,p,z}:=\displaystyle{\frac{r+1}{z}\left((1+q^{-1})p^{-(r+2)}-p^{-2}\right)} ,~\text{for }k\ge 1~\text{and }z>1.\label{10:kk8}
\end{align}
\end{remark}

\section{Bounds for NB Approximation}\label{10:sec3}
In this section, we obtain error bounds between $\mathbb{E}[(\text{N}_{r,p}-z)^+]$ and $\mathbb{E}[(V-z)^+]$ such that $\text{N}_{r,p}$ follows NB distribution and $V=\sum_{i\in J}\zeta_i$, where $\{\zeta_i\}_{i\in J}$ is a collection of $\mathbb{Z}_+$-valued rvs. Throughout this section, let $\mu_X$ and $\sigma_X$ denote the mean and variance for the rv $X$. The following theorem gives the bound for the locally dependent setup. 

\begin{theorem}\label{10:th1}
Let $\mathbb{E}(\zeta_i^3)<\infty$ and $V$ be the sum of locally dependent rvs as defined in \eqref{10:v}. Then
\begin{enumerate}
\item (uniform bound) \quad\quad \quad  $\sup_{z\ge 0}|\mathbb{E}[\mathscr{A}g_z(V)]|\le \left(2p^{-(r+1)}-p^{-1}\right)U_J$
\item (non-uniform bound) \quad $|\mathbb{E}[\mathscr{A}g_z(V)]|\le \vartheta_{r,p,z}U_J$, for all $z>1$,
\end{enumerate}
where 
\begin{align*}
U_J=\left\{
\begin{array}{ll}
\vspace{0.2cm}
\displaystyle{\sum_{i\in J}[p\mathbb{E}(\zeta_i)\mathbb{E}(\zeta_{A_i})+q\mathbb{E}(\zeta_i \zeta_{A_i})+\mathbb{E}(\zeta_i(\zeta_{A_i}-1))]} & \text{if } \mu_{\text{N}_{r,p}}=\mu_V;\\
\vspace{0.07cm}
\displaystyle{p\sum_{i\in J}\mathbb{E}(\zeta_i)\mathbb{E}[\zeta_{A_i}(2\zeta_{B_i}-\zeta_{A_i}-1)\mathscr{D}(V|\zeta_{A_i},\zeta_{B_i})]} & \text{if } \mu_{\text{N}_{r,p}}=\mu_V~\text{and }\\
\vspace{0.07cm}
+\displaystyle{q\sum_{i\in J}\mathbb{E}[\zeta_i\zeta_{A_i}(2\zeta_{B_i}-\zeta_{A_i}-1)\mathscr{D}(V|\zeta_i,\zeta_{A_i},\zeta_{B_i})]} &~~~\sigma_{\text{N}_{r,p}}=\sigma_V.\\
\vspace{0.07cm}
+\displaystyle{\sum_{i \in J}|p\mathbb{E}(\zeta_i)\mathbb{E}(\zeta_{A_i})+q\mathbb{E}(\zeta_i\zeta_{A_i})-\mathbb{E}(\zeta_i(\zeta_{A_i}-1))|\mathbb{E}[\zeta_{B_i}\mathscr{D}(V|\zeta_{B_i})]}\\
\vspace{0.07cm}
+\displaystyle{\sum_{i \in J}\mathbb{E}[\zeta_i(\zeta_{A_i}-1)(2\zeta_{B_i}-\zeta_{A_i}-2)\mathscr{D}(V|\zeta_i,\zeta_{A_i},\zeta_{B_i})]}
 \end{array}\right.
\end{align*} 
and $\vartheta_{r,p,z}$ is defined in \eqref{10:kk8}.
\end{theorem}
\begin{proof}
Consider the Stein operator given in \eqref{10:so} and taking expectation with respect to $V$, we get
\vspace{-0.3cm}
\begin{align}
\mathbb{E}[\mathscr{A}g_z(V)]&=rq\mathbb{E}[g_z(V+1)]+q\mathbb{E}[Vg_z(V+1)]-\mathbb{E}[Vg_z(V)]\nonumber\\
&=p\sum_{i\in J}\mathbb{E}(\zeta_i)\mathbb{E}[g_z(V+1)]+q\sum_{i\in J}\mathbb{E}[\zeta_i g_z(V+1)]-\sum_{i\in J}\mathbb{E}[\zeta_i g_z(V)],\label{10:m0}
\end{align}
where the last expression is obtained by using $\mu_{\text{N}_{r,p}}=\mu_V$. Now, let $V_i=V-\zeta_{A_i}$ then $\zeta_i$ and $V_i$ are independent rvs. Also, note that
\vspace{-0.3cm}
\begin{align}
p\sum_{i\in J}\mathbb{E}(\zeta_i)\mathbb{E}[g_z(V_i+1)]+q\sum_{i\in J}\mathbb{E}[\zeta_i g_z(V_i+1)]-\sum_{i\in J}\mathbb{E}[\zeta_i g_z(V_i+1)]=0.\label{10:m1}
\end{align}
Using \eqref{10:m1} in \eqref{10:m0}, we get
\vspace{-0.3cm}
\begin{align}
\mathbb{E}[\mathscr{A}g_z(V)]&=p\sum_{i\in J}\mathbb{E}(\zeta_i)\mathbb{E}[g_z(V+1)-g_z(V_i+1)]+q\sum_{i\in J}\mathbb{E}[\zeta_i (g_z(V+1)-g_z(V_i+1))]\nonumber\\
&~~~-\sum_{i\in J}\mathbb{E}[\zeta_i( g_z(V)-g_z(V_i+1))]\nonumber\\
&=p\sum_{i\in J}\mathbb{E}(\zeta_i)\mathbb{E}\left[\sum_{j=1}^{\zeta_{A_i}}\Delta g_z(V_i+j)\right]+q\sum_{i\in J}\mathbb{E}\left[\zeta_i \sum_{j=1}^{\zeta_{A_i}}\Delta g_z(V_i+j)\right]\nonumber\\
&~~~-\sum_{i\in J}\mathbb{E}\left[\zeta_i\sum_{j=1}^{\zeta_{A_i}-1}\Delta g_z(V_i+j)\right].\label{10:m2}
\end{align}
Therefore,
\begin{align*}
|\mathbb{E}[\mathscr{A}g_z(V)]|\le \|\Delta g_z\|\sum_{i\in J}[p\mathbb{E}(\zeta_i)\mathbb{E}(\zeta_{A_i})+q\mathbb{E}(\zeta_i \zeta_{A_i})+\mathbb{E}(\zeta_i(\zeta_{A_i}-1))].
\end{align*}
Hence, using Lemma \ref{10:le2}(ii) and \eqref{10:kk8}, the result follows when $\mu_{\text{N}_{r,p}}=\mu_V$.\\
Next, using $\mu_{\text{N}_{r,p}}=\mu_V$ and $\sigma_{\text{N}_{r,p}}=\sigma_V$, it can be easily verified that
\begin{align}
\left[p\sum_{i\in J}\mathbb{E}(\zeta_i)\mathbb{E}\left[\zeta_{A_i}\right]+q\sum_{i\in J}\mathbb{E}\left[\zeta_i \zeta_{A_i}\right]-\sum_{i\in J}\mathbb{E}\left[\zeta_i(\zeta_{A_i}-1)\right]\right]\mathbb{E}[g_z(V+1)]=0.\label{10:m3}
\end{align}
Let $V_i^*=V-\zeta_{B_i}$ then $\zeta_i$ and $\zeta_{A_i}$ are independent of $V_i^*$. Now, using \eqref{10:m3} in \eqref{10:m2}, we get
\begin{align}
\mathbb{E}[\mathscr{A}g_z(V)]&=p\sum_{i\in J}\mathbb{E}(\zeta_i)\mathbb{E}\left[\sum_{j=1}^{\zeta_{A_i}}(\Delta g_z(V_i+j)-\Delta g_z(V_i^*+1))\right]\nonumber\\
&~~~+q\sum_{i\in J}\mathbb{E}\left[\zeta_i \sum_{j=1}^{\zeta_{A_i}}(\Delta g_z(V_i+j)-\Delta g_z(V_i^*+1))\right]\nonumber\\
&~~~-\sum_{i\in J}\mathbb{E}\left[\zeta_i\sum_{j=1}^{\zeta_{A_i}-1}(\Delta g_z(V_i+j)-\Delta g_z(V_i^*+1))\right]\nonumber\\
&~~~-\sum_{i\in J}[p\mathbb{E}(\zeta_i)\mathbb{E}(\zeta_{A_i})+q\mathbb{E}(\zeta_i \zeta_{A_i})-\mathbb{E}(\zeta_i(\zeta_{A_i}-1))]\mathbb{E}[g_z(V+1)-g_z(V_i^*+1)]\nonumber\\
&=p\sum_{i\in J}\mathbb{E}(\zeta_i)\mathbb{E}\left[\sum_{j=1}^{\zeta_{A_i}}\sum_{\ell=1}^{\zeta_{B_i\backslash A_i+j-1}}\Delta^2 g_z(V_i+\ell)\right]+q\sum_{i\in J}\mathbb{E}\left[\zeta_i \sum_{j=1}^{\zeta_{A_i}}\sum_{\ell=1}^{\zeta_{B_i\backslash A_i+j-1}}\Delta^2 g_z(V_i+\ell)\right]\nonumber\\
&~~~-\sum_{i\in J}\mathbb{E}\left[\zeta_i\sum_{j=1}^{\zeta_{A_i}-1}\sum_{\ell=1}^{\zeta_{B_i\backslash A_i+j-1}}\Delta^2 g_z(V_i+\ell)\right]\nonumber\\
&~~~-\sum_{i\in J}[p\mathbb{E}(\zeta_i)\mathbb{E}(\zeta_{A_i})+q\mathbb{E}(\zeta_i \zeta_{A_i})-\mathbb{E}(\zeta_i(\zeta_{A_i}-1))]\mathbb{E}\left[\sum_{\ell=1}^{\zeta_{B_i}}\Delta^2 g_z(V_i+\ell)\right]\nonumber\\
&=p\sum_{i\in J}\mathbb{E}(\zeta_i)\mathbb{E}\left[\sum_{j=1}^{\zeta_{A_i}}\sum_{\ell=1}^{\zeta_{B_i\backslash A_i+j-1}}\mathbb{E}[\Delta^2 g_z(V_i+\ell)|\zeta_{A_i},\zeta_{B_i}]\right]\nonumber\\
&~~~+q\sum_{i\in J}\mathbb{E}\left[\zeta_i \sum_{j=1}^{\zeta_{A_i}}\sum_{\ell=1}^{\zeta_{B_i\backslash A_i+j-1}}\mathbb{E}[\Delta^2 g_z(V_i+\ell)|\zeta_i,\zeta_{A_i},\zeta_{B_i}]\right]\nonumber\\
&~~~-\sum_{i\in J}\mathbb{E}\left[\zeta_i\sum_{j=1}^{\zeta_{A_i}-1}\sum_{\ell=1}^{\zeta_{B_i\backslash A_i+j-1}}\mathbb{E}[\Delta^2 g_z(V_i+\ell)|\zeta_i,\zeta_{A_i},\zeta_{B_i}]\right]\nonumber\\
&~~~-\sum_{i\in J}[p\mathbb{E}(\zeta_i)\mathbb{E}(\zeta_{A_i})+q\mathbb{E}(\zeta_i \zeta_{A_i})-\mathbb{E}(\zeta_i(\zeta_{A_i}-1))]\mathbb{E}\left[\sum_{\ell=1}^{\zeta_{B_i}}\mathbb{E}[\Delta^2 g_z(V_i+\ell)|\zeta_{B_i}]\right].\label{10:m4}
\end{align}
Therefore,
\begin{align*}
|\mathbb{E}[\mathscr{A}g_z(V)]|&\le \|\Delta g_z\|\left\{p\sum_{i\in J}\mathbb{E}(\zeta_i)\mathbb{E}[\zeta_{A_i}(2\zeta_{B_i}-\zeta_{A_i}-1)\mathscr{D}(V|\zeta_{A_i},\zeta_{B_i})]\right.\\
&~~~+q\sum_{i\in J}\mathbb{E}[\zeta_i\zeta_{A_i}(2\zeta_{B_i}-\zeta_{A_i}-1)\mathscr{D}(V|\zeta_i,\zeta_{A_i},\zeta_{B_i})]\\
&~~~+\sum_{i \in J}|p\mathbb{E}(\zeta_i)\mathbb{E}(\zeta_{A_i})+q\mathbb{E}(\zeta_i\zeta_{A_i})-\mathbb{E}(\zeta_i(\zeta_{A_i}-1))|\mathbb{E}[\zeta_{B_i}\mathscr{D}(V|\zeta_{B_i})]\\
&~~~\left.+\sum_{i \in J}\mathbb{E}[\zeta_i(\zeta_{A_i}-1)(2\zeta_{B_i}-\zeta_{A_i}-2)\mathscr{D}(V|\zeta_i,\zeta_{A_i},\zeta_{B_i})]\right\}.
\end{align*}
Hence, using Lemma \ref{10:le2}(ii) and \eqref{10:kk8}, the result follows when $\mu_{\text{N}_{r,p}}=\mu_V$ and $\sigma_{\text{N}_{r,p}}=\sigma_V$.
\end{proof}

\begin{corollary}\label{10:cor2}
Let $V_1=\sum_{i\in J}\zeta_i$ with $p_i=\mathbb{P}(\zeta_i=1)$ and $p_{i,j}=\mathbb{P}(\zeta_i=1,\zeta_j=1)$. Then, for $\mu_{\text{N}_{r,p}}=\mu_{V_1}$, we have
\begin{align}
\sup_{z\ge 0}|\mathbb{E}[\mathscr{A}g_z(V_1)]|\le  \left(2p^{-(r+1)}-p^{-1}\right)\sum_{i\in J}\left[(1+q)\sum_{j\in A_i}p_{i,j}+p_i\left(p\sum_{j\in A_i}p_j-1\right)\right].\label{10:ak15}
\end{align} 
\end{corollary}

\begin{remarks}
\begin{enumerate}
\item[(i)] In Theorem \ref{10:th1}, note that we have the flexibility to choose one parameter (either $r$ or $p$) of our choice when $\mu_{\text{N}_{r,p}}=\mu_V$. Also, the bound is valid only if $\mathbb{E}(V)<\mathrm{Var}(V)$ when $\mu_{\text{N}_{r,p}}=\mu_V$ and $\sigma_{\text{N}_{r,p}}=\sigma_V$. 
\item[(ii)] Observe that $V$ can be expressed as a conditional sum of independent rvs and hence, Subsections 5.3 and 5.4 of R\"{o}llin \cite{RO2008} can be used to obtain the bound of $\mathscr{D}(V|\cdot)$. For more details, see Remarks 3.1(ii) of Kumar {et al.} \cite{KUV}.
\end{enumerate}
\end{remarks}

\noindent
Next, the following theorem gives the bound for independent setup.

\begin{theorem}\label{10:th2}
Let $\mathbb{E}(\zeta_i^3)<\infty$ and $V$ be the sum of independent rvs. Then
\begin{enumerate}
\item (uniform bound) \quad\quad \quad  $\sup_{z\ge 0}|\mathbb{E}[\mathscr{A}g_z(V)]|\le \left(2p^{-(r+1)}-p^{-1}\right)U_J^*$
\item (non-uniform bound) \quad $|\mathbb{E}[\mathscr{A}g_z(V)]|\le \vartheta_{r,p,z}U_J^*$, for all $z>1$,
\end{enumerate}
where 
\begin{align*}
U_J^*=\left\{
\begin{array}{ll}
\vspace{0.3cm}
\displaystyle{\sum_{i\in J}\sum_{k=1}^{\infty}k|(p\mathbb{E}(\zeta_i)+qk)\gamma_{i,k}-(k+1)\gamma_{i,k+1}|} & \text{if } \mu_{\text{N}_{r,p}}=\mu_V;\\
\vspace{0.1cm}
\displaystyle{\sqrt{\frac{2}{\pi}}\left(\frac{1}{4}\hspace{-0.07cm}+\hspace{-0.07cm}\sum_{j\in J}\delta_j\hspace{-0.07cm}-\hspace{-0.07cm}\delta^*\right)^{-\frac{1}{2}}\hspace{-0.1cm}\left\{\sum_{i\in J}\mathbb{E}(\zeta_i)|p\mathbb{E}(\zeta_i)^2+q\mathbb{E}(\zeta_i^2)-\mathbb{E}(\zeta_i(\zeta_{i}-1))|\right.}& \text{if } \mu_{\text{N}_{r,p}}=\mu_V\\
\displaystyle{\left.+\sum_{i\in J}\sum_{k=2}^{\infty}\frac{k(k-1)}{2}|(p\mathbb{E}(\zeta_i)+qk)\gamma_{i,k}-(k+1)\gamma_{i,k+1}|\right\}} & \text{and }\sigma_{\text{N}_{r,p}}=\sigma_V,
 \end{array}\right.
\end{align*} 
$\vartheta_{r,p,z}$ is defined in \eqref{10:kk8}, $\gamma_{i,k}=\mathbb{P}(\zeta_i=k)$, $\delta_j=\min\{\frac{1}{2},1-d_{TV}(\zeta_j,\zeta_j+1)\}$, and $\delta^*=\max_{j \in J}\delta_j$.
\end{theorem}
\begin{proof}
Substituting $A_i=\{i\}$ in \eqref{10:m2}, we get
\begin{align*}
\mathbb{E}[\mathscr{A}g_z(V)]&=p\sum_{i\in J}\mathbb{E}(\zeta_i)\mathbb{E}\left[\sum_{j=1}^{\zeta_{i}}\Delta g_z(V_i+j)\right]+q\sum_{i\in J}\mathbb{E}\left[\zeta_i \sum_{j=1}^{\zeta_{i}}\Delta g_z(V_i+j)\right]\\
&~~~-\sum_{i\in J}\mathbb{E}\left[\zeta_i\sum_{j=1}^{\zeta_{i}-1}\Delta g_z(V_i+j)\right]\\
&=p\sum_{i\in J}\sum_{k=1}^{\infty}\sum_{j=1}^{k}\mathbb{E}(\zeta_i)\mathbb{E}\left[\Delta g_z(V_i+j)\right]\gamma_{i,k}+q\sum_{i\in J}\sum_{k=1}^{\infty}\sum_{j=1}^{k}k\mathbb{E}\left[\Delta g_z(V_i+j)\right]\gamma_{i,k}\\
&~~~-\sum_{i\in J}\sum_{k=2}^{\infty}\sum_{j=1}^{k-1}k\mathbb{E}\left[\Delta g_z(V_i+j)\right]\gamma_{i,k}\\
&=\sum_{i\in J}\sum_{k=1}^{\infty}[(p\mathbb{E}(\zeta_i)+qk)\gamma_{i,k}-(k+1)\gamma_{i,k+1}]\sum_{j=1}^{k}\mathbb{E}\left[\Delta g_z(V_i+j)\right].
\end{align*}
Therefore,
\begin{align*}
|\mathbb{E}[\mathscr{A}g_z(V)]|\le\|\Delta g_z\| \sum_{i\in J}\sum_{k=1}^{\infty}k|(p\mathbb{E}(\zeta_i)+qk)\gamma_{i,k}-(k+1)\gamma_{i,k+1}|.
\end{align*}
\vspace{-0.02cm}
Hence, using Lemma \ref{10:le2}(ii) and \eqref{10:kk8}, the result follows when $\mu_{\text{N}_{r,p}}=\mu_V$.\\
Next, substituting $A_i=B_i=\{i\}$ in \eqref{10:m4}, we get 
\begin{align*}
\mathbb{E}[\mathscr{A}g_z(V)]&=p\sum_{i\in J}\mathbb{E}(\zeta_i)\mathbb{E}\left[\sum_{j=1}^{\zeta_{i}}\sum_{\ell=1}^{j-1}\mathbb{E}[\Delta^2 g_z(V_i+\ell)|\zeta_{i}]\right]+q\sum_{i\in J}\mathbb{E}\left[\zeta_i \sum_{j=1}^{\zeta_{i}}\sum_{\ell=1}^{j-1}\mathbb{E}[\Delta^2 g_z(V_i+\ell)|\zeta_i]\right]\\
&~~~-\sum_{i\in J}[p\mathbb{E}(\zeta_i)^2+q\mathbb{E}(\zeta_i^2)-\mathbb{E}(\zeta_i(\zeta_{i}-1))]\mathbb{E}\left[\sum_{\ell=1}^{\zeta_{i}}\mathbb{E}[\Delta^2 g_z(V_i+\ell)|\zeta_{i}]\right]\\
&~~~-\sum_{i\in J}\mathbb{E}\left[\zeta_i\sum_{j=1}^{\zeta_{i}-1}\sum_{\ell=1}^{j-1}\mathbb{E}[\Delta^2 g_z(V_i+\ell)|\zeta_i]\right]\\
&=p\sum_{i\in J}\sum_{k=1}^{\infty}\sum_{j=1}^{k}\sum_{\ell=1}^{j-1}\mathbb{E}(\zeta_i)\mathbb{E}[\Delta^2 g_z(V_i+\ell)]\gamma_{i,k}+q\sum_{i\in J}\sum_{k=1}^{\infty}\sum_{j=1}^{k}\sum_{\ell=1}^{j-1}k\mathbb{E}[\Delta^2 g_z(V_i+\ell)]\gamma_{i,k}\\
&~~~-\sum_{i\in J}\sum_{k=1}^{\infty}\sum_{\ell=1}^{k}[p\mathbb{E}(\zeta_i)^2+q\mathbb{E}(\zeta_i^2)-\mathbb{E}(\zeta_i(\zeta_{i}-1))]\mathbb{E}[\Delta^2 g_z(V_i+\ell)]\gamma_{i,k}\\
&~~~-\sum_{i\in J}\sum_{k=2}^{\infty}\sum_{j=1}^{k-1}\sum_{\ell=1}^{j-1}k\mathbb{E}[\Delta^2 g_z(V_i+\ell)]\gamma_{i,k}\\
&=\sum_{i\in J}\sum_{k=1}^{\infty}\sum_{j=1}^{k}\sum_{\ell=1}^{j-1}[(p\mathbb{E}(\zeta_i)+qk)\gamma_{i,k}-(k+1)\gamma_{i,k+1}]\mathbb{E}[\Delta^2 g_z(V_i+\ell)]\\
&~~~-\sum_{i\in J}\sum_{k=1}^{\infty}\sum_{\ell=1}^{k}[p\mathbb{E}(\zeta_i)^2+q\mathbb{E}(\zeta_i^2)-\mathbb{E}(\zeta_i(\zeta_{i}-1))]\mathbb{E}[\Delta^2 g_z(V_i+\ell)]\gamma_{i,k}.
\end{align*}
Note that $|\mathbb{E}(\Delta^2 g_z(V_i+\cdot))|\le \delta \|\Delta g_z\|$, where $\delta=2\max_{i\in J}d_{TV}(V_i,V_i+1)$ (see Barbour and Xia \cite{BX}, and Barbour and \v{C}ekanavi\v{c}ius \cite[p. 517]{BC})). Also, from Corollary 1.6 of Brown and Phillips \cite{BP} (see also Remark 4.1 of Vellaisamy {\em et al.} \cite{VUC}), we have $\delta\le \sqrt{\frac{2}{\pi}}\left(\frac{1}{4}+\sum_{j\in J}\delta_j-\delta^*\right)^{-1/2}$ with $\delta_j=\min\{\frac{1}{2},1-d_{TV}(\zeta_j,\zeta_j+1)\}$ and $\delta^*=\max_{j \in J}\delta_j$. Therefore,
\begin{align*}
|\mathbb{E}[\mathscr{A}g_z(V)]|&\le \|\Delta g_z\|  \sqrt{\frac{2}{\pi}}\left(\frac{1}{4}+\sum_{j\in J}\delta_j-\delta^*\right)^{-\frac{1}{2}}\left\{\sum_{i\in J}\mathbb{E}(\zeta_i)|p\mathbb{E}(\zeta_i)^2+q\mathbb{E}(\zeta_i^2)-\mathbb{E}(\zeta_i(\zeta_{i}-1))|\right.\\
&~~~~\left.+\sum_{i\in J}\sum_{k=2}^{\infty}\frac{k(k-1)}{2}|(p\mathbb{E}(\zeta_i)+qk)\gamma_{i,k}-(k+1)\gamma_{i,k+1}|\right\}.
\end{align*}
Hence, using Lemma \ref{10:le2}(ii) and \eqref{10:kk8}, the result follows when $\mu_{\text{N}_{r,p}}=\mu_V$ and $\sigma_{\text{N}_{r,p}}=\sigma_V$.
\end{proof}

\noindent
Next, for $J=\{1,2,\ldots,n\}$, we present and compare our results for the sum of Bernoulli and geometric rvs as special cases. 
\newpage
\begin{remarks}\label{10:re11}
\begin{enumerate}
\item[(i)] Note that the expression $U_J^*$ in Theorem \ref{10:th2} is similar to the expression given in Theorems 3.1 and 4.1 of Vellaisamy {et al.} \cite{VUC}.
\item[(ii)] Let $V_2=\sum_{i=1}^{n}\zeta_i$ be the sum of independent Bernoulli rvs. Then, from Theorem \ref{10:th2}, we have
\begin{align}
\sup_{z\ge 0}|\mathbb{E}[\mathscr{A}g_z(V_2)]|\le \left(2p^{-(r+1)}-p^{-1}\right)\sum_{i=1}^{n}p_i(1-pq_i),\label{10:ak12}
\end{align} 
where $p_i=1-q_i=\mathbb{P}(\zeta_i=1)$ and $r(1-p)=p\sum_{i=1}^{n}p_i$. Note that we can not obtain the bound by matching mean and variance as $\mathbb{E}(V_2)>\mathrm{Var}(V_2)$. From Corollary 1 of Neammanee and Yonghint \cite{NY}, we have
\begin{align}
\sup_{z\ge 0}|\mathbb{E}[\mathscr{A}g_z(V_2)]|\le (2e^\lambda-1)\sum_{i=1}^{n}p_i^2,\label{10:ak11}
\end{align}
where $\lambda=\sum_{i=1}^{n}p_i$. Observe that the bound given in \eqref{10:ak12} is either comparable to or an improvement over the bound given in \eqref{10:ak11}, for example, some numerical comparisons are given in Table \ref{10:tab1}. 
\item[(iii)] Let $V_3=\sum_{i=1}^{n}\zeta_i$ be the sum of independent geometric rvs with $\mathbb{P}(\zeta_i=k)=q_i^k p_i$, for $k\in \mathbb{Z}_+$, and $q_i\le 1/2$. Then, from Theorem \ref{10:th2}, we have
\begin{align}
\sup_{z\ge 0}|\mathbb{E}[\mathscr{A}g_z(V_2)]|\le \left\{
\begin{array}{ll}
\vspace{0.3cm}
\displaystyle{\left(2p^{-(r+1)}-p^{-1}\right)\sum_{i=1}^{n}\frac{|p-p_i|q_i}{p_i^2}} & \text{if } \mu_{\text{N}_{r,p}}=\mu_{V_3};\\
\vspace{0.1cm}
\displaystyle{3\left(2p^{-(r+1)}-p^{-1}\right)\sqrt{\frac{2}{\pi}}\left(\sum_{j=1}^{n}q_j-\frac{1}{4}\right)^{-1/2} } & \text{if } \mu_{\text{N}_{r,p}}=\mu_{V_3} \\
\times \displaystyle{\sum_{i=1}^{n}\frac{|p-p_i|q_i^2}{p_i^3}} & \text{ and }\sigma_{\text{N}_{r,p}}=\sigma_{V_3},
 \end{array}\right.\label{10:lol}
\end{align} 
where  $\sum_{i=1}^{n}q_i> 1/4$ when $\mu_{\text{N}_{r,p}}=\mu_{V_3}$ and $\sigma_{\text{N}_{r,p}}=\sigma_{V_3}$. Note that if $p_i=p$, for all $1\le i\le n$, then $\sup_{z\ge 0}|\mathbb{E}[\mathscr{A}g_z(V_2)]|=0$, as expected. From Theorem 1 and Corollary 2 of Neammanee and Yonghint \cite{NY}, we have
\begin{align}
\sup_{z\ge 0}|\mathbb{E}[\mathscr{A}g_z(V_2)]|\le (2e^\lambda-1)\sum_{i=1}^{n}\frac{(8-7p_i)q_i^2}{p_i^3},\label{10:ny}
\end{align}
where $\lambda=\sum_{i=1}^{n}(q_i/p_i)$. The above bound is better than the bound given by Jiao and Karoui \cite{JK} (shown in Remark 1(1) by Neammanee and Yonghint \cite{NY}). Note that our bound is better than the bound given in \eqref{10:ny}. For instance, let $n=75$ and $q_i$, $1\le i\le 75$, be defined as follows:
 \begin{table}[H]
  \centering
  \caption{The values of $q_i$}
  \label{10:tab2}
  \begin{tabular}{cccccccccccc}
\toprule
$i$ & $q_i$ & $i$ & $q_i$ & $i$ & $q_i$ & $i$ & $q_i$ & $i$ & $q_i$& $i$ & $q_i$\\
\midrule
0-10 & 0.05 &11-20 & 0.10 & 21-30 & 0.15 & 31-40 & 0.20 &  41-50 & 0.25  & 51-75 & 0.30\\
\bottomrule
\end{tabular}
\end{table}

\noindent
Then, choose $r=n$ if $\mu_{\text{N}_{r,p}}=\mu_V$, the following table gives a comparison between our bounds and the existing bounds under Bernoulli and geometric setup.
\begin{table}[H]  
  \centering
  \caption{Comparison of bounds.}
  \label{10:tab1}
  \begin{tabular}{|l|ll|lll|}
\hline
\multirow{3}{*}{$n$} & \multicolumn{2}{c|}{For Bernoulli setup} & \multicolumn{3}{c|}{For geometric setup}\\
\cline{2-6}
& \multirow{2}{*}{From \eqref{10:ak11}}&\multirow{2}{*}{From \eqref{10:ak12}}&\multirow{2}{*}{From \eqref{10:ny}} & From \eqref{10:lol}  & From \eqref{10:lol} \\
&&&& ($\mu_{\text{N}_{r,p}}=\mu_V$) & ($\mu_{\text{N}_{r,p}}=\mu_V$ and $\sigma_{\text{N}_{r,p}}=\sigma_V$)\\
\hline
$10$ & $21.5280$ & $22.2920$ & $0.09390$ & $9.47\times 10^{-17}$ & $9.07\times 10^{-17}$\\
$20$ & $158.986$ & $161.239$ &$2.53041$ & $0.41416$ & $0.06280$ \\
$30$ & $1438.02$ & $1348.40$ &$60.4516$ & $7.17325$ & $1.23534$ \\
$40$ & $22467.0$ & $17633.1$ &$2117.84$ & $195.211$ & $27.7360$ \\
$50$ & $745974$ & $423881$ &$142995$  & $7902.23$ & $1079.63$ \\
\hline
\end{tabular}
\end{table}

\noindent
For large values of $n$, note that our bounds are an improvement over the existing bounds for various values of $q_i$. Moreover, for the geometric setup, the bounds are much sharper than the existing bounds as NB and the sum of geometric rvs consists of similar properties. Also, observe that the bounds computed by matching mean and variance are better than the bounds computed by matching mean only, as expected.
\end{enumerate}
\end{remarks}

\section{An Application to CDO}\label{10:sec4}
The CDO is a financial tool that transfers a pool of assets such as auto loans, credit card debt, mortgages, and corporate debt, among many others, into a product and sold to investors. The assets are divided into several tranches, that is, the set of repayment. Each tranche has various credit quality and risk levels. The primary tranches in CDOs are senior, mezzanine, and equity. The investors can opt for multiple tranches to invest as per their interest. For more details, see Neammanee and Yonghint \cite{NY}, Yonghint et al. \cite{YNC2}, Kumar \cite{k2021}, and reference therein.\\
It is known that the CDO occurs in both, locally dependent and independent setup (see Yonghint {\em et al.} \cite{YNC2} and Neammanee and Yonghint \cite{NY} for more details), and therefore, the results obtained in this paper are useful in applications. Consider the similar type of CDO discussed by Yonghint {\em et al.} \cite{YNC2}. Suppose there are $N$ assets that have a constant recovery rate $R$ then the percentage cumulative loss in CDO up to time $T$ is 
\begin{align}
L(T)=\frac{1-R}{N}\sum_{i=1}^{N}\xi_i,\label{10:ak13}
\end{align}
where $\xi_i={\bf 1}_{\{\tau_i \le T\}}$, $\tau_i$ is the default time of the $i^{\text{th}}$ asset, and ${\bf 1}_{A}$ denotes the indicator function of $A$. The expression in \eqref{10:ak13} can be rewritten as
\begin{align}
\mathbb{E}[(L(T ) - z^*)^+]=\frac{1-R}{N}\mathbb{E}[(V_4-z^*)^+],\label{8:jjj}
\end{align} 
where $z^*=(1-R)z/N>0$ is the attachment or the detachment point of the tranche and $V_4=\sum_{i=1}^{N} \xi_i$. Therefore, the problem is reduced to obtain error bounds for $\mathbb{E}[(V_4-z^*)^+]$, and hence, Corollary \ref{10:cor2} and Remarks \ref{10:re11}(ii) are useful in applications. For more details, we refer the reader to Yonghint {\em et al.} \cite{YNC2}, Kumar \cite{k2021}, and reference therein.\\
Next, we compare our results with the existing results under the locally dependent and independent setup. For the independent setup, Neammanee and Yonghint \cite{NY} gives the bound discussed in \eqref{10:ak11} and, for the locally dependent setup, from Theorem 2 of Yonghint {\em et al.} \cite{YNC2}, we have
\begin{align}
\sup_{z\ge 0}|\mathbb{E}[\mathscr{A}g_z(V^*)]|\le  \left(2e^\lambda-1\right)\sum_{i=1}^{n}\left(\sum_{j\in A_i\backslash \{i\}}p_{i,j}^*+\sum_{j\in A_i}p_i^*p_j^*\right),\label{10:lll}
\end{align} 
where $\lambda=\sum_{i=1}^{n}p_i^*$, $p_{i}^*=\mathbb{P}(\xi_i=1)$, and $p_{i,j}^*=\mathbb{P}(\xi_i=1,\xi_j=1)$. Note that our bound given in \eqref{10:ak15} is better than the bound given in \eqref{10:lll}. For instance, let $r=n$, $p_{i,j}^*=p^*$, $1\le i,j\le n$, $A_i=\{i-1,i,i+1\}$, and $q_i$ as defined in Table \ref{10:tab2}, $1\le i\le 75$, then, the following table gives a comparison between the upper bounds given in \eqref{10:ak12}, \eqref{10:ak11}, \eqref{10:ak15}, and \eqref{10:lll} for different values of $p^*$ and $q_i$.
\begin{table}[H]
  \centering
  \vspace{-0.2cm}
  \caption{Comparison for the locally dependent and independent setup.}
    \vspace{-0.2cm}
  \begin{tabular}{|l|ll|l|ll|}
\hline
\multirow{2}{*}{$n$}& \multicolumn{2}{c|}{For independent setup} &\multicolumn{3}{c|}{For locally dependent setup}\\
\cline{2-6}
 & From \eqref{10:ak11} & From \eqref{10:ak12} & $p^*$ & From \eqref{10:lll} & From \eqref{10:ak15} \\
\hline
$15$ & $64.0726$ & $65.9832$   & \multirow{2}{*}{0.4} & $247.274$ & $206.347$ \\
$35$ & $5747.39$ & $4964.44$   & & $22958.3$ & $15690.1$ \\
\hline
$55$  & $6.79\times 10^6$ & $3.00\times 10^6$ & \multirow{2}{*}{0.7} & $3.37\times 10^7$ & $1.39\times 10^7$\\
$75$ &  $4.49\times 10^{10}$ & $6.22\times 10^{9}$  && $2.31\times 10^{11}$ & $3.00\times 10^{10}$\\
\hline
\end{tabular}
\end{table}
\noindent
For large values of $n$, note that our bounds are better than the existing bounds for various values of $p^*$ and $q_i$.

\section*{Appendix A: Some Useful Inequalities}\label{10:app}
Here we give some inequalities and their proofs that have used in Lemmas \ref{10:le2} and \ref{10:le3}. Recall that $f_z$ is a call function, defined in \eqref{10:cf}, and $\text{N}_{r,p}$ follows the negative binomial distribution, defined in \eqref{10:nb}. The following lemma gives uniform and non-uniform upper bounds for $\mathbb{E} [f_z(\text{N}_{r,p})]=\mathbb{E} [(\text{N}_{r,p}-z)^+]$.

\begin{lemma}\label{10:po1}
The following inequalities hold:
\begin{enumerate}
\item[(i)] $\mathbb{E} [(\text{N}_{r,p}-z)^+]\le \frac{rq}{p}$, for $z\ge 0$.
\item[(ii)] $\mathbb{E} [(\text{N}_{r,p}-z)^+]\le \frac{r(r+1)q^2}{zp^2}$, for $z>1$.
\end{enumerate}
\end{lemma}
\begin{proof}
\begin{enumerate}
\item[(i)] For $z\ge 0$, we have
\begin{align*}
\mathbb{E} [(\text{N}_{r,p}-z)^+]&=\sum_{k=1}^{\infty}(k-z)^+ \binom{r+k-1}{k}p^r q^k\le rp^r\sum_{k=1}^{\infty}\binom{r+k-1}{k-1}q^k=\frac{rq}{p}.
\end{align*}
This proves (i).
\item[(ii)] For $z>1$, we have
\begin{align*}
\mathbb{E} [(\text{N}_{r,p}-z)^+]&=\sum_{k=\ceil{z}}^{\infty}(k-z)\binom{r+k-1}{k}p^rq^k\\
&\le \frac{p^r}{\ceil{z}}\sum_{k=\ceil{z}}^{\infty}r (r+1)\ldots (r+k-1)\frac{(k-z)}{(k-1)!}q^k\\
&\le \frac{p^r}{z}\sum_{k=\ceil{z}}^{\infty}\frac{r (r+1)\ldots (r+k-1)}{ (k-2)!}q^k\\
&\le \frac{r(r+1)p^{r}}{z}\sum_{k=2}^{\infty}\binom{r+k-1}{k-2}q^k=\frac{r(r+1)q^2}{zp^2}.
\end{align*}
This proves (ii).
\end{enumerate}
\end{proof}

\noindent
Next, the following lemma gives some inequalities related to the parameters $r$ and $p$ of $\text{N}_{r,p}$.
\begin{lemma}\label{10:po2}
The following inequalities hold:
\begin{enumerate}
\item[(i)] $\displaystyle{\sum_{j=1}^{\infty}\frac{(r+k)\ldots (r+j+k-1)}{k(k+1)\ldots (j+k-1)}q^j\le p^{-(r+1)}-1,~\text{for }k\ge 1}$.
\item[(ii)] $\displaystyle{\sum_{j=1}^{\infty}\frac{(r+k)\ldots (r+j+k-1)}{(k+1)\ldots (j+k)}q^{j}\le\frac{p^{-r}-1}{rq}-1, ~\text{for all }k\ge 1}$.
\item[(iii)] $\displaystyle{\sum_{j=2}^{\infty}\frac{(r+k+1)\ldots (r+j+k-1)}{(k+1)\ldots (j+k-1)}q^{j}\le \frac{p^{-(r+1)}-1}{r}-q},~\text{for } k \ge 1$.
\item[(iv)] $\displaystyle{\sum_{j=1}^{\infty}\frac{(r+k+1)\ldots (r+j+k)}{k(k+1)\ldots (j+k-1)}q^{j}\le \frac{p^{-(r+2)}-1}{(r+1)q}-1}, ~\text{for }k\ge 2$.
\item[(v)] $\displaystyle{\sum_{j=1}^{\infty}\frac{(r+k)\ldots (r+j+k-1)}{k(k+1)\ldots (j+k)}q^{j}\le \frac{p^{-r}-1}{r(r+1)q^2}-\frac{1}{2}},~\text{for }k\ge 2$.
\end{enumerate}
\end{lemma}
\begin{proof}
Note that, for $k=1$, 
\begin{align*}
\sum_{j=1}^{\infty}\frac{(r+1)\ldots (r+j)}{1.2\ldots j}q^j&=\sum_{j=1}^{\infty}\binom{r+j}{j}q^j=p^{-(r+1)}-1.
\end{align*}
Therefore, the inequality (i) holds for $k=1$. Now, suppose it holds for $k=m$, that is,
\begin{align}
\sum_{j=1}^{\infty}\frac{(r+m)\ldots (r+j+m-1)}{m(m+1)\ldots (j+m-1)}q^j\le p^{-(r+1)}-1.\label{10:rs10}
\end{align}
Observe that
\begin{align*}
\sum_{j=1}^{\infty}\frac{(r+m+1)\ldots (r+j+m)}{(m+1)\ldots (j+m)}q^j&=\sum_{j=1}^{\infty}\frac{m(r+j+m)}{(r+m)(m+j)}\frac{(r+m)\ldots (r+m+j-1)}{m(m+1)\ldots (j+m-1)}q^j\\
&\le\sum_{j=1}^{\infty}\frac{(r+m)\ldots (r+m+j-1)}{m(m+1)\ldots (j+m-1)}q^j\\
&\le p^{-(r+1)}-1\quad \text{(using \eqref{10:rs10}).}
\end{align*}
This implies that the inequality (i) holds for $k=m+1$, and hence it holds for all $k\ge  1$. Following similar steps, the inequalities (ii)-(v) can be easily proved.
\end{proof}

\bibliographystyle{PV}
\bibliography{NBATCF}

\end{document}